\documentclass[11pt]{amsart} 
\author{Andy Hammerlindl}
\title{Center bunching without dynamical coherence}

\usepackage{amsmath}
\usepackage{amssymb}
\usepackage{amsfonts}
\usepackage{amsthm}
\usepackage{graphicx}
\usepackage{setspace}
\usepackage{fourier}
\usepackage[T1]{fontenc}

\newcommand{\Aut}{\operatorname{Aut}}

\newcommand{\heis}{\mathfrak{h}}
\newcommand{\lieg}{\mathfrak{g}}

\newcommand{\bbR}{\mathbb{R}}
\newcommand{\bbC}{\mathbb{C}}
\newcommand{\bbQ}{\mathbb{Q}}
\newcommand{\bbZ}{\mathbb{Z}}

\newcommand{\Es}{E^s}
\newcommand{\Ec}{E^c}
\newcommand{\Eu}{E^u}
\newcommand{\Ecu}{E^{cu}}

\newcommand{\Wcu}{W^{cu}}
\newcommand{\Wcs}{W^{cs}}

\newcommand{\inv}{^{-1}}

\newcommand{\ep}{\epsilon}
\newcommand{\lam}{\lambda}

\newcommand{\gam}{\gamma}

\newcommand{\sig}{\sigma}

\newcommand{\qandq}{\quad \text{and} \quad}

\newcommand{\Diff}{\operatorname{Diff}}

\newcommand{\id}{\operatorname{id}}

\newcommand{\TqM}{T_q M}

\newenvironment{mattwo}
{
\left (
\begin{array}{cc}
}
{
\end{array}
\right )
}
\newenvironment{matthree}
{
\left (
\begin{array}{ccc}
}
{
\end{array}
\right )
}

\newtheorem{thm}{Theorem}

\newtheorem{lemma}[thm]{Lemma}

\theoremstyle{remark}
\newtheorem*{remark} {\bf Remark}

\providecommand{\acknowledgement}{{\noindent\sc Acknowledgements.}\quad}

\begin{document}

\maketitle

\begin{abstract}
    We answer a question of Burns and Wilkinson,
    showing that there are open families of
    volume-preserving partially hyperbolic
    diffeomorphisms
    which are accessible and center bunched
    and neither dynamically coherent nor Anosov.
    We also show in the volume-preserving setting that any diffeomorphism
    which is partially hyperbolic and Anosov may be isotoped
    to a diffeomorphism which is partially hyperbolic and not Anosov.
\end{abstract}


Many partially hyperbolic dynamical systems are ergodic.
One of the most general results in this direction is the following theorem
of K.~Burns and A.~Wilkinson \cite{BW-annals}.

\begin{quote}
    Any $C^2$ volume-preserving, accessible,
    center-bunched, partially hyperbolic
    system is ergodic.
\end{quote}
We define these terms briefly and refer the reader to
\cite{BW-annals,BW-DC} for further details.
A $C^1$ diffeomorphism $f$ on a compact Riemannian manifold $M$
is {\em partially hyperbolic}
if there is an integer $k \ge 1$, a non-trivial splitting of the tangent bundle
\[
    TM = \Es \oplus \Ec \oplus \Eu
      \]
invariant under the derivative $Df$,
and continuous positive functions $\nu,\hat \nu,\gam,\hat \gam$
such that \mbox{$\nu, \hat \nu < 1$} and
\[
        \|Tf^k v^s\| < \nu(x) < \gam(x) < \|Tf^k v^c\| < \hat \gam \inv(x) < \hat
        \nu \inv(x) < \|Tf^k v^u\|
\]
for all $x \in M$ and unit vectors $v^s \in \Es(x)$, $v^c \in \Ec(x)$ and $v^u \in \Eu(x)$.
We say that $f$ is {\em \mbox{center} bunched} if the functions can be chosen
so that $\max\{\nu, \hat \nu\} < \gam \hat \gam$.
Further, $f$ is {\mbox{\em accessible}} if for any two points $x,y \in M$ there
is a path from $x$ to $y$ which is a concatenation of $C^1$ subpaths, each
tangent either to $E^s$ or $E^u$.

The above theorem is a generalization of an earlier result appearing in a
unpublished preprint \cite{BW-better}.
In that preprint,
the partially hyperbolic system has an additional
assumption of \emph{dynamical coherence}
meaning that
there are invariant foliations $\Wcs$ and
$\Wcu$ tangent to the subbundles $\Es \oplus \Ec$ and $\Eu \oplus \Ec$.
To avoid making this assumption, a considerable portion of the proof
in \cite{BW-annals}
explains the definition and construction of ``fake foliations'' which fill the roles
of $\Wcs$ and $\Wcu$ in cases where true foliations do not exist.
This makes the exposition in \cite{BW-annals} much longer and more complicated
than the proof in \cite{BW-better}.

At the time, however, Burns and Wilkinson did not know if such extra effort
was necessary.
There were no known non-dynamically coherent examples which
could not be proven ergodic by simpler means.
This current paper gives such an example
showing
that the fake foliations used in \cite{BW-annals} are necessary to
show ergodicity.

\begin{thm}
    For\ \ $r  \ge  1$, there is an open family $\mathcal{U}$ 
    in the $C^1$ topology of\, $C^r$ volume-preserving diffeomorphisms
    such that each diffeomorphism in $\mathcal{U}$
    is partially hyper\-bolic, accessible, and center bunched
    and is neither dynamically coherent nor Anosov.
\end{thm}
\medskip

To show this, we first define a diffeomorphism $f$ which is Anosov,
partially hyperbolic, center bunched, and not dynamical coherent.
We then deform $f$ to produce a diffeomorphism $g$ which is not Anosov,
but satisfies the other three properties.
By \cite{dolgopyat2003}, there is a open set of diffeomorphisms $\mathcal{U}$
which are $C^1$ close to $g$ and which also
have these properties and are accessible.
This will therefore prove the result.

Define a hyperbolic $3 \times 3$ matrix $A$ with integer entries
such that the eigenvalues $\lam_i$
satisfy
\[
    0 < \lam_2^2 < \lam_1 < \lam_2 < 1 < \lam_3.
      \]
For example,
\[
    A =
        \begin{matthree}
        2 & -3 &  1\\
        -3 &  6 & -2\\
        1 & -2 &  1  \end{matthree}
    .
\]
As the characteristic polynomial of $A$ is irreducible over $\bbQ$,
the splitting field $F/\bbQ$ has a Galois group with an order three subgroup
$\{\id, \sig, \sig^2\}$ where
$\sig$ is a
cyclic permutation such that $\sig(\lam_1) = \lam_2$.

For $i=1,2,3$, let $\heis_i = \langle X_i, Y_i, Z_i \rangle$ be a copy of
the Heisenberg Lie algebra where $[X_i,Y_i] = Z_i$.
Define
\begin{itemize}
    \itemsep0.75em
    \item
    $\lieg = \heis_1 \times \heis_2 \times \heis_3$,
    \item
    $\tilde \Gamma_1 = \bbZ[\lam_1] \times \bbZ[\lam_1] \times \tfrac{1}{2}\bbZ[\lam_1]
    \subset \heis_1$,
    \item
    $\tilde \Gamma = \{ v \times \sig(v) \times \sig^2(v) :
    v \in \tilde \Gamma_1 \} \subset \lieg$,
    \item
    $B_i \in \Aut(\heis_i)$ by $B_i(X_i) = \lam_i X_i$
    and $B_i(Y_i) = \lam_i Y_i$ \\
    (which implies $B_i(Z_i) = \lam_i^2 Z_i$),
    \item
    $B = B_1 \times B_2 \times B_3 : \lieg \to \lieg$,
    \item
    $G = \exp(\lieg)$, and
    \item
    $\Gamma = \exp(\tilde \Gamma)$.
\end{itemize}
Then $B$ defines an Anosov diffeomorphism $f$ of the nilmanifold $G/\Gamma$.

Define a partially hyperbolic splitting for $f$ by
\[
    \Es = \langle Z_1, Z_2, Y_1, X_1 \rangle, \quad
    \Ec = \langle Y_2, X_2 \rangle, \qandq
    \Eu = \langle X_3, Y_3, Z_3 \rangle.
\]
The inequalities on eigenvalues were chosen so that neither this nor any other
partially hyperbolic splitting for $f$ is dynamically coherent.

We now set about deforming $f$ to make a non-Anosov example.
The techniques are similar to those used to construct
volume-preserving examples in \cite{bochi2002, bonatti2000,tahzibi-stably}.
Here we use the following lemma, proven in the appendix,
which will allow us to ensure that center bunching is preserved after the
deformation.

\begin{lemma} \label{spinning-lemma}
    Suppose $f \in \Diff^r(M)$ ($r  \ge  1$) has a dominated splitting,
    i.e.,
    a continuous invariant splitting
    $TM = E'_f \oplus E''_f$
    with continuous functions $\alpha,\beta:M \to \bbR$
    such that
    \[
        \|Df v'_x\| < \alpha(x) < \beta(x) < \|Df v''_x\|
    \]
    for all $x \in M$ and unit vectors
    $v'_x \in E'_f(x)$ and $v''_x \in E''_f(x)$.

    Let $q \in M$ be a fixed point
    and $P \subset \TqM$ a $Df_q$-invariant plane.
    For $\theta \in \bbR$ define
    $R_\theta : \TqM \to \TqM$
    as the rotation by angle\, $\theta$ in the plane $P$
    and
    suppose for some $a>0$ and all $\theta \in [0,a]$
    that
    the linear map $R_\theta \circ Df_q$ has no eigenvalues
    in
    $\{ z \in \bbC : \alpha(q)  \le  |z|  \le  \beta(q) \}.$

    Then,
    there is
    $g \in \Diff^r(M)$ isotopic to $f$
    such that $Dg_q = R_a \circ Df_q$
    and $g$ has a
    dominated splitting $TM = E'_g \oplus E''_g$
    which, for some $n  \ge  1$,
    satisfies
    \[
        \|Dg^n v'_x\| <
        \prod_{k=0}^{n-1} \alpha(g^k(x)) <
        \prod_{k=0}^{n-1} \beta(g^k(x)) < \|Dg^n v''_x\|
    \]
    for all $x \in M$ and unit vectors
    $v'_x \in E'_g(x)$ and $v''_x \in E''_g(x)$.
    If $f$ preserves a smooth volume form,
    one may choose $g$ to preserve the same volume form.

    Moreover, for any $\ep > 0$ and closed set $K \subset M$
    with $q \notin K$, one may define $g$ such that 
    $f|_K = g|_K$ and
    for all $x \in K$
    the splittings
    $E'_f(x) \oplus E''_f(x)$ and 
    $E'_g(x) \oplus E''_g(x)$ are $\ep$-close.
\end{lemma}
\begin{remark}
    Analogous results hold when the fixed point is replaced by a periodic point
    and when the domination is assumed to hold only on an invariant closed
    subset, instead of all of $M$.
    Also, if $f$ has more than one dominated splitting, as is the case for a
    partially hyperbolic diffeomorphism, then the
    deformation $g$ may taken as the same for each splitting.
    These properties can be seen from the proof of the lemma.
\end{remark}
In this specific setting,
choose the plane $P$ as the span of $X_2$ and $Z_3$ at a
fixed point $q$.
Since $\lam_2 \lam_3^2 > \lam_3$, there is an angle $a>0$
such that the map $R_a \circ Df_q$ when restricted to $P$
has two eigenvalues{:}
one slightly greater than one and the other greater than
$\lam_3$.


Applying the lemma, there is a deformation $g$ of $f$ 
and an iterate $n \ge 1$ such
that
\begin{itemize}
    \itemsep0.75em
    \item
    $q$ is a hyperbolic fixed point for $g$ with an unstable
    subspace equal to \\ $\langle X_2, X_3, Y_3, Z_3 \rangle$,
    \item
    $\|Dg^n(v^s)\| < (\lam_1 + \ep)^n$ for unit vectors $v^s \in \Es_g$,
    \item
    $(\lam_2 - \ep)^n < \|Dg^n(v^c)\| < (1 + \ep)^n$ for unit vectors $v^c \in \Ec_g$, and
    \item
    $(\lam_3 - \ep)^n < \|Dg^n(v^u)\|$ for unit vectors $v^u \in \Eu_g$.
\end{itemize}
If $\ep$ is sufficiently small, then $g$ is center bunched.

By the ``moreover'' part of the lemma,
we may take a sequence of diffeomorphisms
$g_k$ such that, except at the fixed point $q$, the splittings
$
    \Eu_{g_k} \oplus \Ec_{g_k} \oplus \Es_{g_k}
$
converge to the splitting for $f$ as $k \to \infty$.
If each $g_k$ was dynamically coherent, then at a point $x  \ne  q$
there would be a sequence of submanifolds tangent to $\Ecu_{g_k}$
converging to a submanifold tangent to $\Ecu_f$.
Since no such submanifold exists for $f$ at $x$, this is a contradiction.
Therefore, we may assume $g$ is not dynamically coherent.
By the same argument, no diffeomorphism $C^1$ close to $g$
is dynamically coherent.  Since stable accessibility is $C^1$-dense
\cite{dolgopyat2003},
by perturbing $g$, one can find an open family of accessible examples as desired.

\bigskip


Using Lemma \ref{spinning-lemma}, we can also give a simple direct proof of the following.

\begin{thm}
    Any volume-preserving 
    Anosov diffeomorphism with a partially hyperbolic splitting may be deformed
    into a volume-preserving partially hyperbolic diffeomorphism which is not Anosov.
\end{thm}
\begin{proof}
    Let $f$ be the Anosov diffeomorphism and assume $f$ has a fixed point $q$.
    (If $f$ has no fixed points, a similar proof will work for a periodic
    orbit.)
    If $Df_q$ has non-real eigenvalues, then there is an invariant plane 
    $P \subset \TqM$ such that $Df_q|_P$ has complex conjugate eigenvalues
    $\lam  \ne  \bar \lam$.  For each $\theta \in \bbR$, $R_\theta \circ Df_q|P$
    has two eigenvalues whose product is $|\lam|^2$.
    If $R_\theta \circ Df_q|P$ has non-real eigenvalues,
    they must have the same modulus as $\lam$.
    Further, for some $a > 0$ the eigenvalues become real.
    Applying Lemma \ref{spinning-lemma}, replace $f$ by a diffeomorphism 
    such that $Df_q$ has real
    eigenvalues on $P$.
    By induction, we may assume that all eigenvalues of $Df_q$ are real.

    Suppose now that $Df_q$ is not diagonalizable.
    Then, there is a plane $P$ such that, with respect to some basis,
    $R_\theta|_P$ and $Df_q|_P$ are given respectively by
    \[
            \begin{mattwo}
            \cos \theta & -\sin \theta \\
            \sin \theta & \cos \theta  \end{mattwo}
        \qandq
            \begin{mattwo}
            \lam & b \\
            0   & \lam   \end{mattwo}
          \]
    where $\lam,b \in \bbR$ and $b  \ne  0$.
    Then, the trace of $R_\theta \circ Df_q|_P$ is given by
    $2 \lam \cos \theta + b \sin \theta$
    and from this one sees that
    there is an arbitrary small $\theta$ (possibly negative)
    such that $R_\theta \circ Df_q|_P$
    has distinct real eigenvalues. 
    Thus, by Lemma \ref{spinning-lemma} and induction, one may assume
    that $Df_q$ is diagonalizable with real eigenvalues.

    Suppose the eigenvalues for $Df_q$ are
    $
        \lam_1^s, \cdots, \lam_k^s,
        \ \lam_1^c, \cdots, \lam_\ell^c,
        \ \lam_1^u, \cdots, \lam_m^u
    $
    in order of increasing modulus
    and with superscripts denoting the bundle in the partially hyperbolic
    splitting to which they are associated.
    If $|\lam_j^c| < 1 < |\lam_{j+1}^c|$ for some $j$,
    Lemma \ref{spinning-lemma} may be applied to the span of the two corresponding
    eigenvectors to produce a diffeomorphism $g$
    where $Dg_q$ has an eigenvalue of modulus one.
    Therefore, we may assume either $|\lam_1^c| > 1$
    or $|\lam_\ell^c| < 1$.
    Without loss of generality, assume the latter.
    Since the product of all of the eigenvalues is equal to one,
    it holds that
    $|\lam_\ell^c \lam_1^u \cdots \lam_m^u| > 1$
    and there is $\mu > 1$ such that
    \[    
        \mu < |\lam_1^u| \qandq \mu^{-m} |\lam_\ell^c \lam_1^u \cdots \lam_m^u| > \mu^m.
    \]
    Now, by applying Lemma \ref{spinning-lemma}
    using the plane associated to 
    $\lam_\ell^c$ and $\lam_i^u$,
    those two eigenvalues may be replaced eigenvalues
    of modulus $|\lam_\ell^c \lam_1^u| \mu$ and $\mu$ respectively.
    By applying similar rotations to replace, in turn,
    each eigenvalue
    $\lam_2^u, \ldots, \lam_m^u$ with $\pm \mu$,
    one produces a new diffeomorphism which has a hyperbolic fixed point
    at $q$ with an $(m+1)$-dimensional unstable direction.
          \end{proof}

\appendix
\section*{Appendix}

We give several technical lemmas before proving Lemma \ref{spinning-lemma}.

\begin{lemma}
    For any $a, \ep>0$ there is a smooth decreasing function
    $\psi:[0,\infty) \to [0,a]$ such that
    \begin{itemize}
        \item $\psi(t) = a$ for $t$ in a neighbourhood of zero,

        \item $\psi(t) = 0$ for $t > \ep$, and

        \item $|t \cdot \psi'(t)| < \ep$ for all $t$.
    \end{itemize}  \end{lemma}
\begin{proof}
    Define
    \[    
        \psi_0(t) = 
        \left\{
            \begin{array}
            {l l}
            a,                        & \text{for } t \in [0,b] \\
            -\frac{\ep}{2} \log(t) + c,  & \text{for } t \in [b,\ep/2] \\
            0,                        & \text{for } t \in [\ep/2,\infty)  \end{array}
        \right.
    \]
    where $b$ and $c$ are chosen so that $\psi_0$ is well-defined
    and continuous.
    By smoothing out $\psi_0$ near $b$ and $\ep/2$, one may define a function
    $\psi$ as desired.
\end{proof}
For the next two lemmas,
let $P$ be a two-dimensional subspace of $\bbR^d$ and define
$R_\theta$ to be the rotation by angle $\theta$ in $P$.

\begin{lemma} \label{lemma-hspin}
    For any $a,\ep>0$ there is a smooth function $h:\bbR^d \to \bbR^d$
    such that at the origin the derivative $Dh_0$ equals $R_a$,
    and for all $x \in \bbR^d$
    \begin{itemize}
        \item there is $\theta \in [0,a]$ such that
        $\|Dh_x - R_\theta\| < \ep$,

        \item if $y \in \bbR^d \setminus \{0\}$ then
        $\|Dh_x - Dh_y\| < (2 + \frac{\|x\|}{\|y\|}) \ep$, and

        \item if $\|x\|>\ep$ then $h(x) = x$.
    \end{itemize}
\end{lemma}
\begin{remark}
    In the proof, the norm of a linear map
    $L:\bbR^d \to \bbR^d$ is taken to be $\sup |\ell_{i j}|$
    where $\ell_{i j}$ is the $(i,j)$ entry of the matrix representing
    $L$ with respect to the standard basis.
    It is easy to see that the same results hold for any 
    choice of norm and any Riemannian metric on $\bbR^d$.
\end{remark}
\begin{proof}
    Define $r(x) = \|x\|$ as the usual Euclidean distance from the origin.
    Then define
    $h(x) = R_{\psi(r(x))}(x)$ where $\psi$ is from the previous lemma.
    From this, one can verify the desired properties.
    For instance,
    if $P$ is the span of the first two coordinates of $\bbR^d$
    so that
    \[    
        R_\theta(x_1, x_2, x_3, \ldots, x_d) =
        (x_1 \cos \theta - x_2 \sin \theta,
         \ x_1 \sin \theta + x_2 \cos \theta,\ x_3,\ldots,x_d)
    \]
    then, writing $r(x)$ simply as $r$,
    \[
        \frac{\partial h_1}{\partial x_1} = 
        \cos \psi(r)
        + [x_1 \sin \psi(r) + x_2 \cos \psi(r)]
        \ \frac{d \psi}{d r}
        \ \frac{\partial r}{\partial x_1}.
    \]
    As $|x_1\ \sin \psi(r) + x_2\ \cos \psi(r)|  \le  r$ and
    $|\frac{\partial r}{\partial x_1}|  \le  1$,
    it follows that
    \[
        \left|\frac{\partial h_1}{\partial x_1}
        - \cos \psi(r)
        \right|
        < r \left|\frac{d \psi}{d r} \right| < \ep.
    \]
    Similar inequalities hold for the other partial derivatives,
    showing that the Jacobian of $h$ at a point $x$ is $\ep$-close to the
    linear map $R_\theta$ where $\theta = \psi(r(x))$.
    By the mean value theorem,
    $|\psi(s)-\psi(t)| < \ep \frac{s}{t}$
    for all $s,t \in (0,\infty)$.
    Therefore,
    \begin{align*}
        \|Dh_x - Dh_y\| & \le 
        \|Dh_x - R_{\psi(r(x))}\| + 
        \|R_{\psi(r(x))} - R_{\psi(r(y))}\| + 
        \|R_{\psi(r(y))} - Dh_y \| \\ &< 
        \ep + \ep \frac{\|x\|}{\|y\|} + \ep
    \end{align*}
    for all non-zero $x,y \in \bbR^d$.
          \end{proof}
\begin{lemma} \label{lemma-gspin}
    Let $f:\bbR^d \to \bbR^d$ be a diffeomorphism such that $f(0)=0$.
    Then for $a,\delta>0$ and $n \ge 1$, there is
    a diffeomorphism $g:\bbR^d \to \bbR^d$
    such that
    at the origin $Dg_0$ equals $R_a \circ Df_0$
    and if $x \in \bbR^d$
    and $j \in \{1, \ldots, 2n\}$ are such that
    $g^j(x)  \ne  f^j(x)$ then
    $\|x\| < \delta$ and there is
    $\theta \in [0,a]$
    such that
    $\|Dg^n_y - (R_\theta \circ Df_0)^n\| < \delta$
    for all \mbox{$y \in \{ x, g(x), \ldots, g^n(x) \}$.}
\end{lemma}
\begin{proof}
    By continuity, there is a constant $\eta>0$ such that
    any linear maps $F_k, H_k:\bbR^d \to \bbR^d$
    and values $\theta \in [0,a]$ and $j \in \{1, \ldots, 2n\}$ 
    which satisfy
    \begin{itemize}
        \item $\| F_k - Df_0 \| < \eta$ for all $k \in \{1, \ldots, 2n\}$,

        \item $\| H_{k+1} - H_{k} \| < \eta$ for all $k \in \{1, \ldots 2n-1\}$, and

        \item $\| H_j - R_\theta \| < \eta$,
    \end{itemize}
    must also satisfy
    $\| (H_{k+n-1} \circ F_{k+n-1} \circ \cdots \circ H_k \circ F_k) - (R_\theta \circ Df_0)^n\| < \delta$
    for all $k \in \{1, \ldots, n\}$.

    Let $U$ be a neighbourhood of the origin such that
    $\| Df_x - Df_0 \| < \eta$
    and $\|x\| < \delta$
    for all $x \in U$.
    Define $K > 1$ such that
    $K \inv \|y\|  \le  \|R_\theta(f(y))\|  \le  K \|y\|$
    for all $\theta \in [0,a]$ and $y \in U$.
    Then, there is $\ep>0$ such that $(2+K)\ep < \eta$ and such that $U$ includes
    the ball of radius $\ep K^{2n}$ centered at the origin.
    With this $\ep$, take $h$ as in the Lemma \ref{lemma-hspin}
    and define $g = h \circ f$.
\end{proof}
Note that the diffeomorphism $h$ in Lemma \ref{lemma-hspin} preserves
the standard volume form on $\bbR^d$.  Therefore, if $f$ is volume preserving
in Lemma \ref{lemma-gspin}, then so is $g = h \circ f$.

\begin{proof}
    [Proof of Lemma \ref{spinning-lemma}]

    By a result of Moser \cite{moser1965}, there is a neighbourhood $U$ of $q$
    and a volume preserving embedding $\phi:U \to \bbR^d$
    such that $q$ is mapped to the origin.
    By abuse of notation, we simply assume that $U$ \emph{is}
    a subset of $\bbR^d$
    and identify the tangent space $T_x M$ with $\bbR^d$ for all $x \in U$.
    We further assume, by changing the embedding if necessary,
    that
    $R_\theta$
    for $\theta \in [0,a]$
    is a rotation with respect to the standard metric on $\bbR^d$.
    Without loss of generality, assume the function $\beta$
    in the statement of the lemma
    is constant in a neighbourhood of $q$.
    
    By considering the spectral radius, one can show that
    for every $\theta \in [0,a]$ there is $n  \ge  1$
    such that the cone
    \[    
        C_\theta =
        \{ v \in \bbR^d : \|(R_\theta \circ Df_q)^n v\|  \ge  \beta(q)^n \|v\| \}
    \]
    satisfies the property that $(R_\theta \circ Df_q)^n(C_\theta)$ is compactly
    contained in $C_\theta$.
    If this inclusion holds for some $n$ and $\theta$, then
    it also holds for the same $n$ and all nearby $\theta$.
    Therefore, as $[0,a]$ is compact, 
    a single value of $n$ may be used.

    Define the cone field $C_f$ by
    \[
        C_f(x) = \{ v \in T_x M : \| Df^n_x v \|  \ge  \beta_n(x) \|v\| \}
    \]
    where $\beta_n$ is the cocycle
    $\beta_n(x) := \beta(f^{n-1}(x)) \cdots \beta(f(x)) \beta(x)$.
    By the properties of cone fields and dominated splittings,
    if $n$ is sufficiently large,
    then $Df^n(C_f)$ is compactly contained in $C_f$ \cite{pujsam2000}.

    For an arbitrary linear map $L:\bbR^d \to \bbR^d$,
    define a cone
    $
        C_L = \{ v \in \bbR^d : \|L v\|  \ge  \beta(q)^n \|v\| \}.
          $
    Suppose $\delta > 0$, $\theta \in [0,a]$,
    and that $L_1$ and $L_2$ are two linear maps
    which satisfy 
    \[
        \|L_i - (R_\theta \circ Df_q)^n\| < \delta.
    \]
    Since $(R_\theta \circ Df_q)^n(C_\theta)$ is compactly
    contained in $C_\theta$,
    a continuity argument shows that if $\delta$ is sufficiently small,
    then 
    $L_1(C_{L_1}) \subset C_{L_2}$.
    Moreover, $\delta$ may be chosen independently of $\theta \in [0,a]$.
    Using $a$, $\delta$, and $n$, define $g$ as in Lemma \ref{lemma-gspin}
    and define a cone field $C_g$ for $g$ by
    \[    
        C_g(x) = \{ v \in T_x M : \| Dg^n_x v \|  \ge  \beta_n(x) \|v\| \}.
    \]
    Now consider $x \in M$.
    If $f^k(x) = g^k(x)$ for all $k \in \{1, \cdots, 2N\}$,
    then
    \[
        Dg^n(C_g(x)) = Df^n(C_f(x)) \subset C_f(f^n(x)) = C_g(g^n(x)).
    \]
    Otherwise, define $L_1 = Dg^n_{x}$
    and $L_2 = Dg^n_y$ where $y={g^n(x)}$.
    Then Lemma \ref{lemma-gspin} implies that
    $\|L_i - (R_\theta \circ Df_q)^n\| < \delta$
    and therefore 
    \[
        Dg^n(C_g(x)) = L_1(C_{L_1}) \subset C_{L_2} = C_g(g^n(x)).
    \]
    This is enough to establish a dominated splitting
    $TM = E'_g \oplus E''_g$.

    For any $\ep > 0$,
    one can show that there is $N_\ep$ such that if
    $f^k(x) = g^k(x)$
    for all $k \in \{1, \ldots, N_\ep\}$
    then the splittings
    $E'_f(x) \oplus E''_f(x)$ and
    $E'_g(x) \oplus E''_g(x)$
    are $\ep$-close.
    This can be used to prove the ``moreover''
    part of the lemma.
\end{proof}

\bigskip

\acknowledgement
The author would like to thank K.~Burns and A.~Wilkinson for reading through
an early draft of this paper.
This research was partially funded by
the Australian Research Council Grant DP$120104514$.


\bibliographystyle{plain}
\bibliography{dynamics}

\end{document}